\documentclass[reqno]{amsart}
\usepackage{hyperref}
\usepackage{amssymb}

\allowdisplaybreaks

\theoremstyle{plain}
\newtheorem{theorem}{Theorem}[section]

\newtheorem{lemma}[theorem]{Lemma}

\theoremstyle{definition}
\newtheorem{definition}[theorem]{Definition}

\theoremstyle{remark}
\newtheorem{remark}[theorem]{Remark}
\numberwithin{equation}{section}
\DeclareMathOperator*{\esssup}{ess\,sup}
\DeclareMathOperator*{\essinf}{ess\,inf}

\begin{document}

\title[Blow-up problem for porous medium equation]
{Blow-up problem for porous medium equation with absorption under nonlinear nonlocal boundary condition}

\author[A. Gladkov]{Alexander Gladkov}
\address{Alexander Gladkov \\ Department of Mechanics and Mathematics
\\ Belarusian State University \\  4  Nezavisimosti Avenue \\ 220030
Minsk, Belarus }    \email{gladkoval@bsu.by}

\subjclass[2020]{ 35K20, 35K61, 35K65}
\keywords{Porous medium equation; nonlocal boundary condition; blow-up; global existence}

\begin{abstract}
In this paper we consider initial boundary value problem 
for porous medium equation with absorption under
nonlinear nonlocal boundary condition and nonnegative initial datum.
We prove local existence, comparison principle, global existence and blow-up of solutions.

\end{abstract}

\maketitle

\section{Introduction}

In this paper we consider the initial boundary value problem for
nonlinear parabolic equation
\begin{equation}\label{v:u}
    u_t= \Delta u^\mu - a u^\nu,\;x\in\Omega,\;t>0,
\end{equation}
with nonlinear nonlocal boundary condition
\begin{equation}\label{v:g}
\frac{\partial u(x,t)}{\partial\bf{n}}=
\int_{\Omega}{k(x,y,t)u^l(y,t)}\,dy, \; x\in\partial\Omega, \; t > 0,
\end{equation}
and initial datum
\begin{equation}\label{v:n}
    u(x,0)=u_{0}(x),\; x\in\Omega,
\end{equation}
where $\mu >1,\,$ $a,\, \nu,\, l$ are positive numbers, $\Omega$ is a bounded domain in $\mathbb{R}^N$
for $N\geq1$ with smooth boundary $\partial\Omega$, $\bf{n} $ is unit
outward normal on $\partial\Omega.$

Throughout this paper we suppose that nonnegative functions
$k(x,y,t)$ and $u_0(x)$ satisfy the following conditions
\begin{equation*}
k(x, y, t) \in L^{\infty}_{loc} (\partial \Omega \times \overline{\Omega} \times [0, \infty)), \; u_0(x) \in L^\infty (\Omega).
\end{equation*}

Various phenomena in the natural sciences and engineering lead to the nonclassical mathematical models subject to nonlocal boundary conditions.
For global existence and blow-up of solutions for parabolic equations and systems
with nonlocal boundary conditions we refer to \cite{F,S,CL,ZK,Lu,YX,YF,MV,FZ3,GN1,KT,GN2,GG2,AD,KD,LZL} and the references therein.
In particular, the blow-up problem for parabolic equations with nonlocal boundary condition
\begin{equation*}
    u(x,t)=\int_{\Omega}k(x,y,t)u^l(y,t)\,dy,\;x\in\partial\Omega,\;t>0,
\end{equation*}
was considered in~\cite{GK,CYZ,FZ2,GG,GG1,GK4,Liu,HJF,G0}. Initial boundary value problems for parabolic equations with
nonlocal boundary condition (\ref{v:g}) were addressed in many papers also (see, for example, \cite{GK1,GK2,LLLW,LWSL,LHZ,G3,G4}).
So, the problem~(\ref{v:u})--(\ref{v:n}) with $\mu = 1$ was studied in~\cite{G1,G2}.
Uniqueness and blow-up problems for porous medium equation with absorption and local nonlinear boundary condition were analyzed in~\cite{AMTR}.

The aim of this paper is to investigate global existence and blow-up of solutions of~(\ref{v:u})--(\ref{v:n}).

This paper is organized as follows. In the next section we prove local existence of solutions. A comparison principle we establish in Section 3. 
General analysis of the blow-up problem we give in the last two sections. The global existence of solutions for any initial data is proved in Section 4. 
In Section 5 we present finite time blow-up results.


\section{Local existence}\label{le}
In this section a local existence of solutions of (\ref{v:u})--(\ref{v:n}) will be proved. 
We begin with definitions of a supersolution, a subsolution and a solution of~(\ref{v:u})--(\ref{v:n}). 
Let $Q_T=\Omega\times(0,T),\,S_T=\partial\Omega\times(0,T), \, T > 0.$
\begin{definition}\label{Sol}
We say that a nonnegative function $u(x,t)\in C([0, T]; L^1 (\Omega)) \cap L^\infty (Q_T)$ is a supersolution
of~(\ref{v:u})--(\ref{v:n}) in $Q_{T}$ if
\begin{equation*}
\int_{\Omega} u(x,t) \varphi (x,t) dx \geq  \int_{\Omega} u_0 (x) \varphi (x,0)  dx  
 + \int_0^t\int_{\partial\Omega} \varphi (x,\tau) \int_\Omega k(x,y,\tau) u^l(y,\tau) dy  dS_x d\tau 
 \end{equation*}
\begin{equation}\label{1.1}
+ \int_0^t\int_{\Omega} \left[ u (x,\tau) \varphi_{\tau} (x,\tau) + u^\mu (x,\tau) \Delta \varphi (x,\tau) - a u^\nu \varphi (x,\tau) \right] 
dx d\tau 
\end{equation}
for every $t \in (0, T]$ and every nonnegative function $\varphi (x,t) \in C^{2,1}(Q_T)\cap C^{1,0}(\overline {Q_T})$ such that
$\varphi_t, \Delta \varphi \in  L^2 (Q_T)$ and  
$\frac{\partial \varphi (x,t)}{\partial \nu} = 0$ for 
$(x,t)\in S_T.$ A nonnegative function $u(x,t)\in  C([0, T]; L^1 (\Omega)) \cap L^\infty (Q_T)$ is called a
subsolution of~(\ref{v:u})--(\ref{v:n}) in $Q_{T}$ if 
it satisfies~(\ref{1.1}) in the reverse order. We say that $u(x,t)$ is a solution of
problem~(\ref{v:u})--(\ref{v:n}) in $Q_T$ if $u(x,t)$ is both a
subsolution and a supersolution of~(\ref{v:u})--(\ref{v:n}) in $Q_{T}.$
\end{definition}

From \cite{GSib,B} we immediately infer 

\begin{lemma}\label{L1}
There exists a sequence of positive functions $u_{0m} (x) \in L^\infty (\Omega),\,$ $m \in \mathbb{N},$ 
possessing the following properties: 
\begin{equation*}
u_{0(m+1)} (x) \le u_{0m} (x) \,\, \textrm{and} \,\, u_{0m} (x) \to u_{0} (x)  \,\, \textrm{as} \,\, m \to \infty
 \,\, \textrm{almost everywhere (a.e.) in} \,\,  \Omega,
 \end{equation*}
\begin{equation*}
\frac{1}{m} \le u_{0m} (x) \,\, \textrm{a.e. in} \,\,  \Omega.
\end{equation*}
For every $m \in \mathbb{N}$  there is a sequence of positive functions $u_{0mj} (x) \in C (\Omega),\,$ $j \in \mathbb{N},$ 
 satisfying the conditions 
\begin{equation*}
u_{0(m+1)j} (x) \le u_{0mj} (x)  \le u_{0m(j+1)} (x) \,\, \textrm{in} \,\,  \Omega,
\end{equation*}
\begin{equation*}
 u_{0mj} (x) \to u_{0m} (x)  \,\, \textrm{as} \,\, j \to \infty
 \,\, \textrm{a.e. in} \,\,  \Omega,
\end{equation*}
\begin{equation*}
\frac{1}{m} \le u_{0mj} (x)  \,\, \textrm{in} \,\,  \Omega.
\end{equation*}

\end{lemma}

\begin{theorem}\label{LE} 
 Problem (\ref{v:u})--(\ref{v:n}) has a solution in $Q_{T}$ for small values of $T.$
\end{theorem}
\begin{proof}

From Lemma~\ref{L1} it is easy to get that
\begin{equation*} \label{l0}
 u_{0mj} (x) \le \esssup_{\Omega} u_{01} (x), \, m,j \in \mathbb{N}.
\end{equation*}
We put $u_{m0} (x, t) \equiv 0$ and consider for $m,j \in \mathbb{N}$ the following initial boundary value problem
\begin{eqnarray} \label{l1}
\left\{ \begin{array}{ll}
 L_m u_{mj} \equiv u_{mjt} - \Delta u_{mj}^\mu + a u_{mj}^\nu  - a /m^\nu = 0
 \,\,\,&\textrm{for} \,\,\, (x,t) \in Q_T, \\
\frac{\partial u_{mj}(x,t)}{\partial\bf{n}} = \int_{\Omega} k(x,y,t) u_{m(j-1)}^l (y,t) dy \,\,\,
& \textrm{for} (x,t) \in S_T,  \\
u_{mj} (x,0)= u_{0mj} (x) \,\,\,& \textrm{for} \,\,\, x \in \Omega.
\end{array} \right.
\end{eqnarray}
It is well known that problem (\ref{l1}) has a classical solution.

Let us consider the following auxiliary function 
\begin{equation*} \label{l2}
w (x, t) = \left[ 1 - \alpha (\mu - 1) t \right]^{- \frac{1}{\mu - 1}} \zeta (x),
\end{equation*}
where
\begin{equation*} \label{l3}
\zeta (x) \in C^2(\overline\Omega), \; \inf_{\Omega} \zeta (x) \ge \max \{1, \esssup_{\Omega} u_{01} (x) \}, \;  
\alpha > \sup_{\Omega} \frac{|\Delta \zeta^\mu |}{\zeta},
\end{equation*}
\begin{equation*} \label{l4}
\inf_{\partial \Omega} \frac{\partial \zeta (x)}{\partial \bf{n}} \geq  
\max \{1, 2^{(l-1)/(\mu -1)} \} \esssup_{\partial\Omega \times \Omega \times [0, 1/\{ 2 \alpha (\mu - 1) \}]} k (x, y, t)  \int_{\Omega} \zeta^l(y) \, dy.  
\end{equation*}
It is easy to check that $\underline{u} (x,t) = 1/m$ and $w (x, t)$ are subsolution and supersolution of (\ref{l1}) for $j =1$ 
in $Q_T$ with $T = 1/[2 \alpha (\mu - 1)],$ respectively. By a comparison principle for (\ref{l1}) we have
\begin{equation*} \label{l5}
\frac{1}{m} \le u_{m1} (x, t) \le w (x, t)\,\,\, \textrm{in} \,\,\, Q_T.
\end{equation*}
Then using the induction on $j$ and a comparison principle for (\ref{l1}) we deduce
\begin{equation} \label{l6}
\frac{1}{m} \le u_{mj} (x, t) \le w (x, t)\,\,\, \textrm{in} \,\,\, Q_T, j \in \mathbb{N}.
\end{equation}
Obviously, 
\begin{equation*} 
u_{m1} (x, t) \ge u_{m0} (x, t)\,\,\, \textrm{in} \,\,\, Q_T \,\,\, \textrm{for} \,\,\, m \in \mathbb{N}.
\end{equation*}
Let
\begin{equation}\label{l7} 
u_{mj} (x, t) \ge u_{m(j-1)} (x, t) \,\,\, \textrm{in} \,\,\, Q_T \,\,\, \textrm{for} \,\,\, m \in \mathbb{N} \,\,\, \textrm{and for some} \,\,\, j \in \mathbb{N}.
\end{equation}
Using (\ref{l1}), (\ref{l7}) and Lemma~\ref{L1}  we obtain
\begin{equation*} 
L_m u_{m(j+1)} (x, t) = L_m u_{mj} (x, t) = 0 \,\,\, \textrm{in} \,\,\, Q_T,
\end{equation*}
\begin{equation*} 
\frac{\partial u_{m(j+1)}(x,t)}{\partial\bf{n}} \ge \frac{\partial u_{mj}(x,t)}{\partial\bf{n}} \,\,\, \textrm{on} \,\,\, S_T,
\end{equation*}
\begin{equation*} 
 u_{0m(j+1)} (x) \ge  u_{0mj} (x)  \,\, \textrm{in} \,\,  \Omega.
\end{equation*}
Now by a comparison principle  we have
\begin{equation*}\label{l8} 
u_{m(j+1)} (x, t) \ge u_{mj} (x, t) \,\,\, \textrm{in} \,\,\, Q_T \,\,\, \textrm{for} \,\,\, m, j \in \mathbb{N}.
\end{equation*}

We note that
\begin{equation*} 
L_m u_{m1} (x, t) = 0, \; L_m u_{(m+1)1} (x, t) =  L_{m+1} u_{(m+1)1} (x, t)  - \frac{a}{m^\nu}  + \frac{a}{(m + 1)^\nu} 
 \le 0 \,\,\, \textrm{in} \,\,\, Q_T,
\end{equation*}
\begin{equation*} 
\frac{\partial u_{(m+1)1}(x,t)}{\partial\bf{n}} = \frac{\partial u_{m1}(x,t)}{\partial\bf{n}} = 0 \,\,\, \textrm{on} \,\,\, S_T,
\end{equation*}
\begin{equation*} 
 u_{0(m+1)1} (x) \le  u_{0m1} (x)  \,\, \textrm{in} \,\,  \Omega.
\end{equation*}
Then by a comparison principle  we obtain
\begin{equation*} 
u_{(m+1)1} (x, t) \le u_{m1} (x, t) \,\,\, \textrm{in} \,\,\, Q_T \,\,\, \textrm{for} \,\,\, m \in \mathbb{N}.
\end{equation*}
In a similar way as above using the induction on $j$ and a comparison principle we deduce
\begin{equation}\label{l9} 
u_{(m+1)j} (x, t) \le u_{mj} (x, t) \,\,\, \textrm{in} \,\,\, Q_T \,\,\, \textrm{for} \,\,\, m, j \in \mathbb{N}.
\end{equation}

 Multiplying the first equation in (\ref{l1}) by $\varphi (x,t)$ from Definition~\ref{Sol} and then integrating over $Q_t$  for $t \in (0, T],$ we get
\begin{eqnarray*}
\int_{\Omega} u_{mj} (x,t) \varphi (x,t) \, dx &=& \int_0^t\int_{\Omega} \left[ u_{mj} \varphi_{\tau}  + u_{mj}^\mu  \Delta \varphi - a u_{mj}^\nu \varphi 
+  \frac{a}{m^\nu} \varphi  \right] \, dx d\tau \nonumber\\
&+& \int_0^t\int_{\partial\Omega} \varphi (x,\tau) \int_\Omega k(x,y,\tau) u_{m(j-1)}^l(y,\tau) \, dy  dS_x d\tau \nonumber\\
&+&   \int_{\Omega} u_{0mj} (x) \varphi (x,0) \, dx.
\end{eqnarray*}
 Since the sequence $u_{mj} (x, t)$ is monotone on $j$ and bounded, we may define 
 \begin{equation}\label{l91}
 u_{m} (x, t) = \lim_{j \to \infty} u_{mj} (x, t),
\end{equation}
and it is easy to see that $u_{m} (x, t)$ satisfies the following equation
\begin{eqnarray}\label{l10}
\int_{\Omega} u_{m} (x,t) \varphi (x,t) \, dx &=& \int_0^t\int_{\Omega} \left[ u_{m} \varphi_{\tau}  + u_{m}^\mu  \Delta \varphi - a u_{m}^\nu \varphi 
+  \frac{a}{m^\nu} \varphi  \right] \, dx d\tau
\nonumber\\
&+& \int_0^t\int_{\partial\Omega} \varphi (x,\tau) \int_\Omega k(x,y,\tau) u_{m}^l(y,\tau) \, dy  dS_x d\tau 
\nonumber\\
&+&   \int_{\Omega} u_{0m} (x) \varphi (x,0) \, dx.
\end{eqnarray}

Moreover, from (\ref{l6}), (\ref{l9}), (\ref{l91}) we have
\begin{equation}\label{l12} 
\frac{1}{m} \le u_{m} (x, t) \le w (x, t), \,\,\,  u_{m+1} (x, t) \le u_{m} (x, t) \,\,\, \textrm{in} \,\,\, Q_T \,\,\, \textrm{for} \,\,\, m \in \mathbb{N}.
\end{equation}
Now we define
\begin{equation}\label{l13} 
u (x, t) = \lim_{m \to \infty} u_{m} (x, t)
\end{equation}
and prove that  $u(x,t)\in C([0, T]; L^1 (\Omega)).$ To show this we integrate the first equation in
 (\ref{l1})  over $Q_t$  for $t \in (0, T]$ to obtain
\begin{eqnarray}\label{l15} 
\int_{\Omega} u_{mj} (x,t) \, dx &=& \int_0^t\int_{\Omega} \left[ \frac{a}{m^\nu} - a u_{mj}^\nu  \right] \, dx d\tau \nonumber\\
&+& \int_0^t\int_{\partial\Omega} \int_\Omega k(x,y,\tau) u_{m(j-1)}^l(y,\tau) \, dy  dS_x d\tau \nonumber\\
&+&   \int_{\Omega} u_{0mj} (x) \, dx.
\end{eqnarray}
Substracting from (\ref{l15}) the similar equality with $m = k \, (k > m)$ and taking (\ref{l9}) into account, we get 
\begin{eqnarray}\label{l16} 
&&\int_{\Omega} [ u_{mj} (x,t) -  u_{kj} (x,t) ] \, dx \le a T |\Omega| \left( \frac{1}{m^\nu} - \frac{1}{k^\nu} \right) +  \int_{\Omega} [ u_{0mj} (x) -  u_{0kj} (x) ]  \, dx
 \nonumber\\
&+& |\partial\Omega| \esssup_{\partial\Omega \times \Omega \times [0, T]} k (x, y, t)  \int_0^T \int_\Omega [ u_{m(j-1)}^l(y,\tau) - u_{k(j-1)}^l(y,\tau) ]\, dy d\tau, 
\end{eqnarray}
where $\vert \partial \Omega\vert$ and $\vert \Omega\vert $ are
the Lebesgue measures of $\partial \Omega$ in $\mathbb R^{N-1}$ and
$\Omega$ in $\mathbb R^N,$ respectively. Passing to the limit in (\ref{l16}) as $j \to \infty,$ 
by virtue of (\ref{l6}), (\ref{l91}) -- (\ref{l13}) and  Lemma~\ref{L1}, we infer 
\begin{equation*}\label{l17} 
\lim_{m \to \infty} \sup_{[0, T]} || u_{m} (x, t) - u_{k} (x, t) ||_{L^1 (\Omega)} = 0.
\end{equation*}
Thus, $u_{m} $ is a Cauchy sequence in $C([0, T]; L^1 (\Omega))$ and the limit function $u $ is continuous in 
$L^1 (\Omega).$ 
Now passing to the limit in (\ref{l10}) as $m \to \infty,$ we prove that $u(x,t)$ is a solution of~(\ref{v:u})--(\ref{v:n}) in $Q_{T}.$ 
\end{proof}


\section{Comparison principle}\label{cp}

In this section a comparison principle
for~(\ref{v:u})--(\ref{v:n}) will be proved.
\begin{theorem}\label{Th3} Let $\overline{u}$ and $\underline{u}$ be a
 supersolution and a  subsolution of problem
(1.1)--(1.3) in $Q_T,$ respectively.  Suppose that $\max \{\underline{u}(x,t),\overline{u}(x,t) \} \ge \delta $ 
for some positive constant $\delta,$   a.e. in $Q_T$ if $ l < 1.$ Then $ \overline{u}(x,t) \geq
\underline{u}(x,t) $ a.e. in ${Q}_T.$
\end{theorem}
\begin{proof}
Suppose that $l \geq 1.$  Let $u_m(x,t)$ is defined in  (\ref{l91}). Then it satisfies (\ref{l12}). 
 To establish theorem we will show that
\begin{equation}\label{3.1}
\underline{u}(x,t) \leq u(x,t) \leq \overline{u}(x,t) \,\,\,
 \textrm{a.e. in} \,\,\, Q_T,
 \end{equation}
where $u(x,t)$ was defined in  (\ref{l13}). 
We prove the second inequality in (\ref{3.1}) only, since the
proof of the first one is similar. Let $\varphi (x,t) \in
C^{2,1}(\overline{Q_T})$ be a nonnegative function such that
$$
\frac{\partial \varphi (x,t)}{\partial \bf{n}} = 0, \; (x,t) \in
S_T.
$$
Set
$w(x,t)=u_m (x,t) - \overline{u}(x,t).$ Then $w(x,t)$
satisfies
\begin{eqnarray}\label{5.1}
\int_\Omega w(x,t)\varphi (x,t)\, dx &\leq&
\int_\Omega w(x,0)\varphi (x,0)\, dx +  \frac{a}{m^\nu} \int_0^t\int_{\Omega}  \varphi (x,\tau)  \, dx d\tau 
\nonumber\\
&+& \int_0^t\int_{\partial\Omega} \varphi (x,\tau) \int_\Omega
k(x,y,\tau) d_m (y,\tau) w(y,\tau) \, dy dS_x d\tau 
 \nonumber \\
 &+& \int_0^t\int_{\Omega} (\varphi_{\tau} + a_m \Delta \varphi - a b_m \varphi) w \, dx d\tau,
\end{eqnarray}
where 
\begin{eqnarray*}\label{6.1}
a_m = 
\left\{ \begin{array}{ll}
\frac{u_m^\mu - \overline{u}^\mu}{u_m - \overline{u}}, \, u_m \neq \overline{u},
 \\
 \null \\
\mu u_m^{\mu-1}, \, u_m = \overline{u},
\end{array} \right.
b_m = 
\left\{ \begin{array}{ll}
\frac{u_m^\nu - \overline{u}^\nu}{u_m - \overline{u}}, \, u_m \neq \overline{u},
 \\
 \null \\
\nu u_m^{\nu-1}, \, u_m = \overline{u},
\end{array} \right.
d_m = 
\left\{ \begin{array}{ll}
\frac{u_m^l - \overline{u}^l}{u_m - \overline{u}}, \, u_m \neq \overline{u},
 \\
 \null \\
l u_m^{l-1}, \, u_m = \overline{u}.
\end{array} \right.
\end{eqnarray*}
Note here that by
hypotheses for  $k(x,y,t)$, $u_m (x,t)$ and
$\overline{u}(x,t)$, we have
\begin{eqnarray}\label{7.1}
&&  0 \leq \overline{u}(x,t) \leq M,
\,\, \frac{1}{m} \leq u_m (x,t) \leq M, 
\nonumber \\
&& r_m \leq a_m (x,t) \leq M, \,\, r_m \leq d_m (x,t) \leq M, \,\, r_m \leq b_m (x,t) \leq M_m \,\,
\textrm{a. e. in} \,\,  Q_T, 
\nonumber \\
&&\textrm{and}\,\,\,0 \leq k(x,y,t) \leq M \,\,\, \textrm{a. e. in}\,\,\, 
\partial \Omega \times Q_T,
\end{eqnarray}
where $M,\,$ $r_m,\,$ $M_m\,$ are some positive constants, $r_m\,$ and $M_m\,$ may depend on $m.$ 
Let $\{a_{mk}\}, \{b_{mk}\}$ be sequences of functions with the following properties: 
$\, a_{mk} \in C^\infty (\overline{Q_T}),\,$ $\, b_{mk} \in C^\infty (\overline{Q_T}),\,$
\begin{equation}\label{8.0}
a_{mk} \to a_{m} \,\,\, \textrm{as}\,\,\,  k \to \infty  \,\,\, \textrm{in}\,\,\,  L^2({Q_T}),  \,\,\, 
b_{mk} \to b_{m}\,\,\, \textrm{as}\,\,\,   k \to \infty \,\,\, \textrm{in}\,\,\,  L^1({Q_T})
 \end{equation}
and
\begin{equation}\label{8.1}
r_m \leq a_{mk} (x,t) \leq M+1, \,\, r_m \leq b_{mk} (x,t) \leq M_m +1 \,\,
\textrm{ in} \,\,  \overline{Q_T}.
 \end{equation}

Now consider a backward problem given by
\begin{eqnarray}\label{9.1}
\left\{ \begin{array}{ll}
\varphi_{\tau} + a_{mk} \Delta \varphi - a b_{mk} \varphi = 0 \,\,\, & \textrm{for} \,\,\,x \in \Omega, \,\,\,
0<\tau<t, \\
\frac{\partial \varphi (x,\tau)}{\partial  \bf{n}} = 0 \,\,\,&
\textrm{for} \,\,\, x \in
\partial\Omega, \,\,\, 0 < \tau < t, \\
\varphi(x,t)= \psi (x) \,\,\,& \textrm{for} \,\,\,x \in \Omega,
\end{array} \right.
\end{eqnarray}
where $\psi (x)\in C_0^\infty (\Omega)$ and $ 0 \leq \psi (x) \leq
1.$ Denote a solution of (\ref{9.1}) as $\varphi_{mk} (x,\tau).$
Then by the standard theory for linear parabolic equations (see
\cite{LSU}, for example), we find that  $\varphi_{mk} \in
C^{2,1}(\overline{Q}_{t}),$  $0 \leq \varphi_{mk} (x,\tau) \leq 1$ in
$\overline{Q_t}.$  Putting $\varphi = \varphi_{mk}$ in
(\ref{5.1}), we infer
\begin{eqnarray}\label{10.1}
&& \int_\Omega w(x,t)\psi (x)\, dx \leq \int_\Omega w(x,0)_+ \, dx +  \frac{a}{m^\nu} T \vert \Omega \vert  
 +  \vert \partial\Omega \vert M^2  \int_0^t \int_\Omega w(y,\tau)_+ \, dy d\tau
 \nonumber\\
&+&  \int_0^t \int_\Omega \left\{ (a_{m} - a_{mk}) \Delta \varphi_{mk} - a  (b_{m} - b_{mk}) \varphi_{mk} \right\}  w(x,\tau) \, dx d\tau,
\end{eqnarray}
where $w_+ = \max \{w,0 \}.$ 

To estimate last integral in the right-hand side of (\ref{10.1}) 
 we multiply the equation in (\ref{9.1}) by $\Delta \varphi_{mk}$ and integrate the result over $Q_t$ 
\begin{eqnarray}\label{11.1}
 \int_0^t\int_{\Omega} a_{mk} (\Delta \varphi_{mk})^2 \, dx d\tau &=& - \int_0^t\int_{\Omega} \varphi_{mk\tau} \Delta \varphi_{mk} \, dx d\tau 
 + a \int_0^t\int_{\Omega} b_{mk} \varphi_{mk} \Delta \varphi_{mk} \, dx d\tau
  \nonumber\\
&\leq& \frac{1}{2} \int_0^t\int_{\Omega} \vert \nabla \psi (x) \vert^2 \, dx 
+ \frac{a^2 }{2}  \int_0^t\int_{\Omega} \frac{b_{mk}^2}{a_{mk}} \varphi_{mk}^2 \, dx d\tau 
 \nonumber\\
 &+& \frac{1}{2} \int_0^t\int_{\Omega} a_{mk} (\Delta \varphi_{mk})^2 \, dx d\tau. 
\end{eqnarray}
From  (\ref{11.1}) we conclude that 
 \begin{equation}\label{12.1}
\int_0^t\int_{\Omega} a_{mk} (\Delta \varphi_{mk})^2 \, dx d\tau \leq C_m,
\end{equation}
 where $C_m$ is some positive constant which may depend on $m.$ 
 Taking into account~(\ref{7.1}) -- (\ref{8.1}), (\ref{12.1}) and 
	H{$\ddot o$}lder's  inequality, we obtain
$$
\Big\vert \int_0^t \int_\Omega \left\{ (a_{m} - a_{mk}) \Delta \varphi_{mk} - a  (b_{m} - b_{mk}) \varphi_{mk} \right\}  w(x,\tau) \, dx d\tau \Big\vert
$$
$$
\leq M \left( \int_0^t \int_\Omega \frac{(a_{m} - a_{mk})^2}{a_{mk}}  w^2   \, dx d\tau \right)^\frac{1}{2}
\left( \int_0^t\int_{\Omega} a_{mk} (\Delta \varphi_{mk})^2 \, dx d\tau \right)^\frac{1}{2}
$$
 \begin{equation*}\label{13.1}
+ a M \int_0^t \int_\Omega \vert  b_{m} - b_{mk} \vert \, dx d\tau  \to 0 \,\,\, \textrm{as}\,\,\,  k \to \infty.
\end{equation*}
Now passing in (\ref{10.1}) to the limit as $k \to \infty,$  we
infer
 \begin{equation}\label{14.1}
\int_\Omega w(x,t)\psi (x)\, dx \leq \int_\Omega w(x,0)_+ \, dx +  \frac{a}{m^\nu}  T \vert \Omega \vert  
 +  \vert \partial\Omega \vert M^2  \int_0^t \int_\Omega w(y,\tau)_+ \, dy d\tau.
\end{equation}
Since (\ref{14.1}) holds for every $\psi (x),$ we can choose a sequence $\{ \psi_n \}$
converging in $L^1(\Omega)$ to $\psi (x) =1$ if $w(x,t) > 0$ and $\psi
(x) = 0$ otherwise. Hence, from (\ref{14.1}) we get
\begin{equation*}
\int_\Omega w(x,t)_+ \, dx \leq \int_\Omega w(x,0)_+ \, dx +  \frac{a}{m^\nu}  T \vert \Omega \vert  
 +  \vert \partial\Omega \vert M^2  \int_0^t \int_\Omega w(y,\tau)_+ \, dy d\tau.
\end{equation*}
Applying now Gronwall's inequality and passing to the limit as
$m \to \infty,$ the conclusion of the theorem follows for  $l \geq 1.$
For the case $l < 1$ we can consider
$w(x,t)=\underline{u}(x,t) - \overline{u}(x,t)$ and prove the
theorem in a similar way using the positiveness of a subsolution
or a supersolution.
\end{proof}

\begin{remark}
It is not difficult to show that a classical subsoluiton and a classical supersoluiton of (\ref{v:u})--(\ref{v:n}) 
are also a subsoluiton and a supersoluiton, respectively.

\end{remark}



\section{Global existence}\label{gl}

To formulate global existence result we need the following condition
\begin{equation*}\label{K}
|| k(x, y, t) ||_{L^\infty (\partial\Omega \times \overline{\Omega} \times [0, \infty))} = K_\infty < \infty.
\end{equation*}

\begin{theorem}\label{global}
Let at least one of the following conditions hold:

1) $l + \mu \leq 2,$

2) $\nu > \mu + l - 1,$

3) $l + \mu > 2, \,$  $\nu = \mu + l - 1,\,$ $ a / K_\infty $ is large enough.

Then every solution of (\ref{v:u})--(\ref{v:n}) is global.
\end{theorem}
\begin{proof}
In order to prove global existence of solutions we construct a
suitable explicit supersolution of~(\ref{v:u})--(\ref{v:n}) in
$Q_T$ for any positive $T.$ Suppose at first that $l + \mu < 2.$
Let
\begin{equation*}\label{k}
K_T =  || k(x, y, t) ||_{L^\infty (\partial\Omega \times \overline{\Omega} \times [0, T])}.
\end{equation*}
We now construct a supersolution of~(\ref{v:u})--(\ref{v:n}) in $Q_T$ as follows
\begin{equation}\label{E1111}
\overline u (x,t) = \left\{ (1 - l) \left[ \psi  (x) + \left( \alpha t + \beta \right)^\frac{1-l}{2-l-\mu } \right] \right\}^\frac{1}{1-l},
\end{equation}
where positive constants $\alpha, \beta$ will be chosen later and $\psi (x)$ is some positive solution of the following problem 
\begin{equation}
    \begin{cases}\label{E111}
  \Delta \psi (x) = b, \; x \in \Omega, \\
        \frac{\partial \psi (x)}{\partial\bf{n}} = \frac{b |\Omega|}{|\partial \Omega|}, \; x \in \partial \Omega
    \end{cases}
\end{equation}
with $b > 0.$ Due to (\ref{E1111}), (\ref{E111}) we have
\begin{align}\label{Eq1}
       L \overline u &\equiv \overline u_t - \Delta\overline u^\mu + a\overline u^\nu 
       \geq \frac{\alpha (1 - l)}{2- l - \mu} \overline u^{\mu + l - 1} 
       \frac{\left( \alpha t + \beta \right)^\frac{\mu-1}{2-l-\mu } }{\overline u^{\mu - 1}}
       \nonumber \\
       &- \mu b \overline u^{\mu + l - 1} - \mu (\mu + l -1) |\nabla \psi|^2 \overline u^{\mu + 2l - 2} \geq 0
      \end{align}
in $Q_T$ for large values of $\alpha$ and $\beta$  and 
\begin{equation}\label{E2}
 \frac{\partial \overline u}{\partial\nu} = \frac{b |\Omega|}{|\partial \Omega|} {\overline u^l} \geq \int_{\Omega} k(x,y,t)\overline u^l(y,t) \,dy
 \end{equation}
 on $S_T$ for large values of $b$ and $\beta.$ At last,
\begin{equation}\label{E3}
\overline u(x,0)\geq u_0(x) \,\,\, \textrm{ a.e. in} \,\,\, \Omega
\end{equation}
for  $\beta$  large enough. By virtue of (\ref{Eq1})--(\ref{E3}) and comparison principle 
the solution of (\ref{v:u})--(\ref{v:n}) exists globally.

For $l + \mu = 2$ it is easy to check that 
\begin{equation}\label{E4}
\overline u (x,t) = \left\{ (1 - l) \left[ \psi  (x) + \beta \exp{(\alpha t)} \right] \right\}^\frac{1}{1-l},
\end{equation}
is a supersolution of~(\ref{v:u})--(\ref{v:n}) in $Q_T$ for large values of 
$\alpha$ and $\beta.$ 

Suppose now that  $\nu > \mu + l - 1$ and $l < 1.$ Then function in (\ref{E4}) with $\alpha =0$ 
is a supersolution of~(\ref{v:u})--(\ref{v:n}) in $Q_T$ for $\beta$ large enough.
 
Let  $\nu > \mu + l - 1$ and $l \ge 1.$ 
To construct a supersolution we use the change of variables in a
	neighborhood of $\partial \Omega$ as in \cite{CPE}. Let
	$\overline x\in\partial \Omega$ and $\widehat{n}
	(\overline x)$ be the inner unit normal to $\partial \Omega$ at the
	point $\overline x.$ Since $\partial \Omega$ is smooth it is well
	known that there exists $\delta >0$ such that the mapping $\psi
	:\partial \Omega \times [0,\delta] \to \mathbb{R}^n$ given by
	$\psi (\overline x,s)=\overline x +s\widehat{n} (\overline x)$
	defines new coordinates ($\overline x,s)$ in a neighborhood of
	$\partial \Omega$ in $\overline\Omega.$ A straightforward
	computation shows that, in these coordinates, $\Delta$ applied to
	a function $g(\overline x,s)=g(s),$ which is independent of the
	variable $\overline x,$ evaluated at a point $(\overline x,s)$ is
	given by
	\begin{equation}\label{E5}
	\Delta g(\overline x,s)=\frac{\partial^2g}{\partial s^2}(\overline x,s)-\sum_{j=1}^{n-1}\frac{H_j(\overline x)}{1-s
		H_j (\overline x)}\frac{\partial g}{\partial s}(\overline x,s),
	\end{equation}
	where $H_j (\overline x)$ for $j=1,...,n-1,$ denote the principal
	curvatures of $\partial\Omega$ at $\overline x.$ For $0\leq s\leq \delta$
	and small $\delta$  we have
	\begin{equation}\label{E6}
	\left\vert\sum_{j=1}^{n-1} \frac{H_j (\overline x)}{1-s H_j (\overline
		x)}\right\vert\leq\overline c.
	\end{equation}
	
	Let $ \rho > 0, \,$ $0<\varepsilon<\omega<\min(\delta \rho, 1), \,$
	$\max\{\mu/l, 2\mu/(\nu - \mu)\} < \beta, \,$ $\beta < 2\mu/(l-1)$ for $l > 1,$  $0<\gamma<\beta/2,$ $A^\mu \ge \esssup_\Omega u_0(x).$
	For points in $Q_{\delta,T}=\partial \Omega \times [0,
	\delta]\times [0,T]$ of coordinates $(\overline x,s,t)$ define
	\begin{equation}\label{E7}
	\overline u (x,t)=  \overline u ((\overline x,s),t) = 
\left( \left[(\rho s + \varepsilon)^{-\gamma}-\omega^{-\gamma}\right]_+^\frac{\beta}{\gamma} + A \right)^\frac{1}{\mu},
	\end{equation}
	where $s_+=\max(s,0).$ For points in $\overline{Q_T}\setminus Q_{\delta,T}$
	we set  $ \overline u(x,t)= A.$ We prove that $ \overline u (x,t)$
	is the supersolution of~(\ref{v:u})--(\ref{v:n}) in $Q_T.$
	It is not difficult to check that
	\begin{equation}\label{E8}
	\left\vert\frac{\partial \overline u^\mu}{\partial s}\right\vert\leq \rho \beta \min\left(\left[ D(s)\right]^\frac{\gamma+1}{\gamma}\left[( \rho s + \varepsilon)^{-\gamma}-\omega^{-\gamma}\right]_+^\frac{\beta+1}{\gamma},\,(\rho s + \varepsilon)^{-(\beta+1)}\right),
	\end{equation}
	\begin{equation}\label{E9}
	\left\vert\frac{\partial^2 \overline u^\mu}{\partial s^2}\right\vert\leq \rho^2 \beta (\beta+1)\min\left(\left[D(s)\right]^\frac{2(\gamma+1)}{\gamma}\left[(\rho s+\varepsilon)^{-\gamma}-
	\omega^{-\gamma}\right]_+^\frac{\beta+2}{\gamma},\,(\rho s+\varepsilon)^{-(\beta+2)}\right),
	\end{equation}
	where
	\begin{equation*}
	D(s)= \frac{(\rho s+\varepsilon)^{-\gamma}}{ (\rho s+\varepsilon)^{-\gamma}-\omega^{-\gamma}}.
	\end{equation*}
	Then $D^\prime(s)>0$ and for any $\overline\varepsilon>0$
	\begin{equation}\label{E10}
	1\leq D(s)\leq 1+\overline\varepsilon, \; 0<s\leq{\overline s},
	\end{equation}
	where ${\overline s} = ( [\overline\varepsilon/(1+\overline\varepsilon)]^{1/\gamma}\omega-\varepsilon)/ \rho,\,$
	$\varepsilon<[\overline\varepsilon/(1+\overline\varepsilon)]^{1/\gamma}\omega.$
By  (\ref{E5})--(\ref{E10})  we can choose
	$\overline\varepsilon$ small so that in $Q_{{\overline s},T}$
\begin{align*}
        L\overline u &\geq   a \left( \left[(\rho s+\varepsilon)^{-\gamma}-\omega^{-\gamma}\right]_+^\frac{\beta }{\gamma} + A \right)^\frac{\nu}{\mu} - \rho^2 \beta(\beta+1)\left[ D(s)\right]^\frac{2(\gamma+1)}{\gamma}\left[(\rho s + \varepsilon)^{-\gamma}-\omega^{-\gamma}\right]_+^\frac{\beta+2}{\gamma} \\
       &- \rho \beta \overline c \left[ D(s)\right]^\frac{\gamma+1}{\gamma} \left[(\rho s + \varepsilon)^{-\gamma}-\omega^{-\gamma}\right]_+^\frac{\beta+1}{\gamma} \geq 0.
      \end{align*}
	
	Let $s\in[{\overline s},\delta].$ From (\ref{E5})--(\ref{E9}) we have
	\begin{equation*}
	\vert\Delta \overline u^\mu \vert \leq \rho^2 \beta(\beta+1)\omega^{-(\beta+2)}\left(\frac{1+\overline\varepsilon}
	{\overline\varepsilon}\right)^\frac{\beta+2}{\gamma} + \rho \beta\overline c\omega^{-(\beta+ 1)}\left(\frac{1+\overline\varepsilon}{\overline\varepsilon}\right)^\frac{\beta+1}{\gamma}
	\end{equation*}
	and  $L\overline u\geq0$ for $A$ large enough. Obviously, in $\overline{Q_T}\setminus Q_{\delta,T}$
	\begin{equation*}
	L\overline u = a A^\frac{\nu}{\mu}  \geq  0.
	\end{equation*}
		
	Now we prove the following inequality
\begin{equation}\label{E:4.6}
\frac{\partial \overline u}{\partial\nu} (\overline x,0,t) \geq
\int_{\Omega} K_T \overline u^l(\overline x,s,t) \, dy, \quad
(x,t) \in S_T
\end{equation}
for a suitable choice of $\varepsilon.$ To estimate the integral
$I$ in the right hand side of (\ref{E:4.6}) we use the change of
variables in a neighborhood of $\partial\Omega$ as above. Let
	\begin{equation*}
 \overline J= \sup_{0< s< \delta} \int_{\partial\Omega}
\vert J(\overline y,s)\vert \, d\overline y,
	\end{equation*}
where $J(\overline y,s)$ is Jacobian of the change of variables.
Then we have
\begin{align*}
I \leq & \theta  K_T \int_{\Omega} \left[ (\rho s +
\varepsilon)^{-\gamma} - \omega^{-\gamma}\right]_+^\frac{\beta
l}{\gamma\mu} \, dy + \theta  K_T A^\frac{l}{\mu} \vert\Omega\vert\\
\leq & \theta  K_T  \overline J \int_{0}^{(\omega - \varepsilon)/\rho} \left[ (\rho s + \varepsilon)^{-\gamma} -
\omega^{-\gamma}\right]^\frac{\beta
l}{\gamma \mu} \, ds + \theta  K_T A^\frac{l}{\mu} \vert\Omega\vert\\
\leq & \frac{\mu \theta  K_T  \overline J}{\rho (\beta l - \mu)} \left[
\varepsilon^{-\left(\frac{\beta l}{\mu}-1 \right)} - \omega^{-\left(\frac{\beta l}{\mu} - 1 \right)}\right] +
\theta  K_T A^\frac{l}{\mu} \vert\Omega\vert,
 \end{align*}
where $\theta = \max (2^{l/\mu-1}, 1). $   On the other hand, since
	\begin{equation*}
\frac{\partial \overline u}{\partial\nu} (\overline x,0,t) = -
\frac{\partial \overline u}{\partial s} (\overline x,0,t) = 
\frac{\rho \beta}{\mu} \varepsilon^{-\gamma -1} \left[ \varepsilon^{-\gamma} -
\omega^{-\gamma}\right]^\frac{\beta-\gamma}{\gamma} 
\left( \left[\varepsilon^{-\gamma}-\omega^{-\gamma}\right]^\frac{\beta}{\gamma} + A \right)^\frac{1 - \mu}{\mu},
	\end{equation*}
the inequality (\ref{E:4.6}) holds if $\varepsilon$ is small
enough. At last,
	\begin{equation*}
 u_0(x) \leq \overline u (x,0) \,\,\,  \textrm{ a.e. in } \,\,\,  \Omega.
	\end{equation*}
Hence, by comparison principle we get
	\begin{equation*}
  u(x,t) \leq \overline u (x,t) \,\,\,  \textrm{ a.e. in } \,\,\,
 \overline{Q}_T.
	\end{equation*}

Suppose now that  $l + \mu > 2, \,$  $\nu = \mu + l - 1\,$ and $l < 1.$ Then function in (\ref{E4}) with $\alpha = 0$ 
is a supersolution of~(\ref{v:u})--(\ref{v:n}) in $Q_T$ for suitable choice of $b \,$ and $\beta$ if $a/K_\infty > \mu |\partial \Omega|.$

For $l  = 1$ and  $\nu = \mu + l - 1 = \mu$ it is not difficult to check that 
\begin{equation*}
\overline u (x,t) = \left[ \psi  (x)  \right]^\frac{1}{\mu},
\end{equation*}
is a supersolution of~(\ref{v:u})--(\ref{v:n}) in $Q_T$ for  $a/K_\infty > \mu |\partial \Omega|$ under suitable choice of
$b$ and $\psi  (x).$ 

For $ l > 1$ and  $\nu = \mu + l - 1$ we can show in the same way as above that $ \overline{u}$ in (\ref{E7}) with  $\beta = 2\mu/(l-1)$ 
is a supersolution of~(\ref{v:u})--(\ref{v:n}) in $Q_T$  under suitable choice of
$\rho, \,$ $\varepsilon \,$ and $A$ if
\begin{equation*}
\frac{a}{K_\infty} > \frac{ \theta\mu (2 \mu + l -1)\overline J}{l + 1}.
\end{equation*}

\end{proof}


\section{Blow-up in finite time}\label{blow}

To formulate finite time blow-up result we suppose that for some $\tau > 0$
\begin{equation}\label{E11}
\essinf_{\partial \Omega \times \overline{\Omega} \times [0, \tau]}  k(x, y, t) = k_0 > 0.
\end{equation}

\begin{theorem}\label{blow-up}
Let  (\ref{E11}) hold, $ l + \mu > 2$ and either $\nu < \mu + l - 1$  or $\nu = \mu + l - 1$ and $a/k_0$ be small enough.
Then there exist solutions of (\ref{v:u})--(\ref{v:n}) with finite time blow-up.
\end{theorem}
\begin{proof}
 To prove the theorem  we construct a suitable subsolution with finite time blow-up and use a comparison argument. 
 Suppose at first that $ l \leq 1$ and  $\nu < \mu + l - 1.$ Then there exists $\gamma > 0$ such that 
$ \gamma < l,$ $\nu < \mu + \gamma - 1$  and $ \gamma + \mu > 2.$ 
We put
	\begin{equation}\label{E110}
 \underline{u} (x,t) = \left\{ (1 - \gamma) [\psi (x) + (T - t)^{-\alpha} + A] \right\}^\frac{1}{1 - \gamma},
	\end{equation}
where $A > 0,\,$ $T \in (0, \tau], \,$ $ \alpha > (1 - \gamma)/(\gamma + \mu - 2)$ 
and $\psi (x)$ was defined in (\ref{E111}). It is easy to check that
\begin{equation}\label{E12}
L \underline u  \leq \alpha (1 - \gamma)^{-\frac{\alpha + 1 }{\alpha}} \underline u^\frac{\alpha + 1 - \gamma}{\alpha} 
- \mu b \underline u^{\mu + \gamma - 1} + a \underline u^\nu  \le 0   \,\,\,  \textrm{ in } \,\,\, Q_T 
\end{equation}
for large values of $A.$ For $x \in \partial \Omega$ and $t \in (0, T)$ we have
\begin{equation}\label{E13}
  \frac{\partial \underline u (x,t)}{\partial\bf{n}} = \frac{b |\Omega|}{|\partial \Omega|} \underline u^\gamma  (x,t)
\leq k_0 \int_{\Omega}{\underline u^l (y,t)}\,dy 
\end{equation}
for $A$  large enough since $ \gamma < l.$ By  (\ref{E11}) --  (\ref{E13}) and comparison principle 
we conclude that any solution $u (x, t)$ of~(\ref{v:u})--(\ref{v:n}) blows up in finite time if   
	\begin{equation}\label{E139}
u_{0}(x) \geq \underline u (x,0) \,\,\,  \textrm{ a.e. in } \,\,\,  \Omega.
	\end{equation}
If $ l < 1$ and  $\nu = \mu + l - 1$ then in a similar way we show that $ \underline{u}$ in (\ref{E110}) 
with $\gamma = l$ is a subsolution of~(\ref{v:u})--(\ref{v:n}) in $Q_\sigma, \, \sigma \in (0, T)$ 
for  $a/k_0 < \mu |\partial \Omega|$ provided (\ref{E139}) holds and $A, \,$ $b$ are appropriately chosen.

For $ l = 1$ and  $\nu = \mu + l - 1 = \mu$ we introduce
\begin{equation}\label{E130}
 \underline{u} (x,t) = B (T - t)^{-\alpha} \exp [ \psi (x) ],
	\end{equation}
where $B \ge 1, \,$ $ \alpha > 1/(\mu - 1),\,$ $\psi (x)$ satisfies (\ref{E111}). Then 
\begin{equation}\label{E131}
L \underline u  \leq \alpha (T - t)^{- 1}  \underline{u} - b \mu \underline{u}^\mu + a \underline{u}^\mu \leq 0
 \,\,\,  \textrm{ in } \,\,\, Q_T 
\end{equation}
for $b > a/\mu$ and large values of $B.$  For $x \in \partial \Omega$ and $t \in (0, T)$ we obtain
\begin{equation}\label{E132}
  \frac{\partial \underline u (x,t)}{\partial\bf{n}} = \frac{b |\Omega|}{|\partial \Omega|} \underline u (x,t)
\leq k_0 \int_{\Omega}{\underline u (y,t)}\, dy 
\end{equation}
if
\begin{equation}\label{E133}
b \leq k_0 \frac{ |\partial \Omega| \int_{\Omega} \exp [ \psi (y) ] \,dy }{|\Omega| \sup_{\partial \Omega} \exp [ \psi (x) ]}.
\end{equation}
So, by (\ref{E111}), (\ref{E11}), (\ref{E130})  --  (\ref{E133}) and comparison principle 
we conclude that for $a/k_0 < \mu |\partial \Omega|$  under suitable choice of $ \psi (x), \,$
$b, \,$ $B$  any solution $u (x, t)$ of~(\ref{v:u})--(\ref{v:n}) blows up in finite time if (\ref{E139}) holds.

Let $l > 1, \,$  $\nu < \mu + l - 1$  and $u(x,t)$ be defined in (\ref{l13}). It is not difficult to check that the function
\begin{eqnarray*}
w_\nu (t) = \left\{ \begin{array}{ll} \left[ A^{1-\nu}  - 2 (1-\nu) a t \right]^{1/(1-\nu)}
\,\,\,&\textrm{for} \,\,\, 0 < \nu <1,\\
A \exp \left( - 2 a t \right)
\,\,\,&\textrm{for} \,\,\, \nu = 1
\end{array} \right.
\end{eqnarray*}
is a subsolution of (\ref{l1})  in $Q_{\tau}$ for $\tau \leq 1/(2a)$ and  large values of $m$ if
\begin{equation*}\label{E14}
 u_0(x) \geq A > 1
\end{equation*}
as well as the function
$$
w_\nu (t) = [2 (\nu-1) a (t + t_0)]^{-\frac{1}{\nu-1}}, \, t_0 > 0,  \,\,\, \textrm{for} \,\,\, \nu > 1
$$
is a subsolution of (\ref{l1})  in $Q_{\tau}$ for $m$  large enough if
\begin{equation*}\label{E15}
 u_0(x) \geq [2 (\nu-1) a t_0]^{-\frac{1}{\nu-1}}.
\end{equation*}
Applying a comparison principle to (\ref{l1}), we have
\begin{equation} \label{E150}
w_\nu (t) \le u_{mj} (x, t) \,\,\, \textrm{in} \,\,\, Q_{\tau}, \, j \in \mathbb{N}
\end{equation}
for large values of $m.$ 
Then from (\ref{l91}), (\ref{l13}), (\ref{E150}) we obtain
\begin{equation}\label{E151}
w_\nu (t) \leq  u(x,t)  \,\,\, \textrm{ in} \,\,\,  Q_{\tau}.
\end{equation}
Now we use the change of variables in a neighborhood of
$\partial\Omega$ as in Theorem~\ref{global}. Set $\Omega_\gamma =
\{ (\overline x,s): \overline x \in \partial\Omega, 0<s<\gamma
\}.$ 

Let us consider the following initial boundary value problem:
\begin{eqnarray}\label{E:4.13}
\left\{
\begin{array}{ll}
v_{t}=\Delta v^\mu - a v^\nu  \,\,\, \,\,\,& \textrm{for} \,\,\,
x \in \Omega_\gamma, \,\,\, 0 < t < t_0,   \\
\frac{\partial v(x,t)}{\partial \bf{n}} = \int_{\Omega_\gamma}
k(x,y,t) v^l (y,t) \, dy \,\,\,\,\,\,\,\,& \textrm{for}\,\,\, x
\in \partial\Omega, \,\,\,   0 < t < t_0, \\
v(x,t)= u(x,t) \,\,\,\,\,\,& \textrm{for}\,\,\, x \in
\partial\Omega_\gamma\setminus\partial\Omega, \,\,\,  0 < t < t_0, \\
v(x,0)= u_0 (x) \,\,\,\,\,\,& \textrm{for}\,\,\, x \in
\Omega_\gamma,
\end{array} \right.
\end{eqnarray}
where $\gamma$ and $t_0 \leq \tau $ will be chosen later. We can define the notions of a
supersolution and a subsolution of (\ref{E:4.13}) in a similar way
as in Definition~\ref{Sol}.

We define
\begin{equation}\label{E:4.14}
\xi(s,t) = C (t_0 + s-t)^{-\sigma},
\end{equation}
where $C > 0, \,$ $\sigma >2/(l-1)$ for $\nu \le \mu$ and $2/(l-1) < \sigma < 2/(\nu - \mu)$ for $\nu > \mu.$

We show that $\xi(s,t)$ is a subsolution of (\ref{E:4.13}) in $Q(\gamma,t_0)$ under suitable choice of
$t_0$ and $\gamma.$
Using (\ref{E5}), (\ref{E6}), (\ref{E:4.14}) we find
\begin{equation}\label{E:4.15}
\begin{split}
-\xi_t + \Delta \xi^\mu - a \xi^\nu &\geq (t_0  +
s-t)^{-\sigma \mu - 2} \left\{ \sigma \mu (\sigma \mu + 1) C^\mu
- \sigma (t_0 + \gamma)^{\sigma (\mu - 1) + 1} C \right.
 \\
 & - \sigma \mu \overline{c} ( t_0 + \gamma) C^\mu - a ( t_0 + \gamma)^{\sigma (\mu - \nu) + 2} C^\nu
 \left. \right\} \geq 0
\end{split} 
\end{equation}
in $Q(\gamma,t_0)$ if we take $\gamma$ and $t_0$ small
enough. Next we check the inequality on the boundary $\partial\Omega \times (0,  t_0).$ Let
$$
\underline J= \inf_{0< s< \gamma} \int_{\partial\Omega}
|J(\overline y,s)| \, d\overline y.
$$
In view of (\ref{E:4.14}) we have
\begin{equation}\label{E:4.16}
\begin{split}
&\frac{\partial \xi}{\partial \nu} (0,t) - k_0 \int_{\Omega_\gamma}
 \xi^l (s,t) \, dy \leq \sigma C (t_0 - t)^{-\sigma -1} -
k_0 \underline J C^l
\int_0^\gamma (t_0 + s - t)^{-\sigma l} ds   
\\
\leq& \sigma C (t_0 - t)^{-\sigma -1} - k_0 \underline J C^l \frac{(t_0
- t)^{-\sigma l + 1}}{\sigma l - 1} \left[ 1 - \left(
\frac{t_0}{t_0 + \gamma} \right)^{\sigma l - 1} \right] \leq 0
\end{split}
\end{equation}
for $x \in \partial\Omega, \, 0<t<t_0$ and small enough $t_0.$
Let
\begin{equation}\label{E16}
 u_0(x) \geq \gamma^{-\sigma}
\end{equation}
and
\begin{equation}\label{E17}
C \gamma^{-\sigma} \leq  \left[ A^{1-\nu} - 2 (1-\nu) a t_0 \right]^{1/(1-\nu)} \,\,\, \textrm{for} \,\,\, 0 < \nu <1,
\end{equation}
\begin{equation}\label{E18}
C \gamma^{-\sigma} \leq A \exp \left( - 2 a t_0 \right) \,\,\, \textrm{for} \,\,\,  \nu = 1,
\end{equation}
\begin{equation}\label{E19}
C \gamma^{-\sigma} \leq [4 (\nu-1) a t_0]^{-\frac{1}{\nu-1}} \,\,\, \textrm{for} \,\,\,  \nu > 1.
\end{equation}

Making use of (\ref{E151}), (\ref{E16}) -- (\ref{E19}) we obtain
\begin{equation}\label{E21}
\xi(s,t) \leq  u(x,t) \,\,\, \textrm{for} \,\,\, x \in
\Omega_\gamma,\, t=0 \,\,\,\textrm{and} \,\,\, x \in
\partial\Omega_\gamma\setminus\partial\Omega, \,\,\, 0 < t < t_0.
\end{equation}
From (\ref{E:4.15}),  (\ref{E:4.16}),  (\ref{E21}) we conclude that $\xi(s,t)$ is a subsolution of (\ref{E:4.13}) in
$Q(\gamma,t_0).$ It is easy to check that $u(x,t)$ is a supersolution of (\ref{E:4.13}) in
$Q(\gamma,t_0).$ Arguing as in Theorem~\ref{Th3}, we prove $u(x,t) \ge \xi(s,t)$ in
$Q(\gamma,t_0).$ Thus,  $u(x,t)$ blows up in finite time since $\xi(0,t) \to \infty$ as $t \to t_0.$

In the case $ l > 1$ and  $\nu = \mu + l - 1$ we show in a similar way that $ \underline{u}$ in (\ref{E:4.14}) with  $\sigma = 2/(l-1)$ 
is a subsolution of~(\ref{v:u})--(\ref{v:n}) in $Q(\gamma,t_0)$ under suitable choice of
$t_0, \,$ $\gamma,\,$ $ C $ and $u_{0}(x)$ if
\begin{equation*}
\frac{a}{k_0} < \frac{\mu (2 \mu + l -1)\underline J}{l + 1}.
\end{equation*}

\end{proof}

\subsection*{Acknowledgements}
This work is supported by the state program of fundamental research of Belarus
(grant 1.2.02.2).

\end{document}